\documentclass{amsart}

\usepackage{hyperref}
\hypersetup{
    colorlinks=true,
    linkcolor=blue,
    citecolor=blue,
    urlcolor=blue,
}

\makeatletter
\newcommand\org@maketitle{}
\newcommand\@authors{}
\let\org@maketitle\maketitle
\def\maketitle{%
	\let\@authors\authors
	\nxandlist{; }{ and }{; }\@authors
	\hypersetup{
		linktocpage=true,
		pdftitle={\@title},
                pdfauthor={\@authors},
                pdfsubject={\subjclassname. \@subjclass},
		pdfkeywords={\@keywords}
	}%
	\org@maketitle
}
\makeatother

\usepackage{amssymb, accents}
\usepackage{rsfso}
\usepackage{mathtools}
\usepackage{microtype}
\usepackage[alphabetic,msc-links]{amsrefs}
\usepackage{enumitem}
\usepackage{amsfonts}
\usepackage{todonotes}
\DeclareMathAlphabet{\mathcal}{OMS}{cmsy}{m}{n}

\usepackage{doi}
\renewcommand{\PrintDOI}[1]{\doi{#1}}

\numberwithin{equation}{section}

\newtheorem{theorem}{Theorem}[section]
\newtheorem{lemma}[theorem]{Lemma}

\newtheorem{proposition}[theorem]{Proposition}
\theoremstyle{definition}
\newtheorem{definition}[theorem]{Definition}
\newtheorem{condition}[theorem]{Condition}

\theoremstyle{remark}
\newtheorem{remark}[theorem]{Remark}

\newcommand{\cD}{{\mathcal D}}

\newcommand{\al}{\alpha}
\newcommand{\be}{\beta}
\newcommand{\g}{\gamma}
\newcommand{\de}{\delta}
\newcommand{\e}{\varepsilon}

\newcommand{\la}{\lambda}
\newcommand{\ka}{\kappa}

\newcommand{\R}{\mathbb{R}}

\newcommand{\bff}{\mathbf{f}}
\newcommand{\bq}{\mathbf{q}}

\newcommand{\D}{\nabla}

\renewcommand{\div}{\operatorname{div}}

\newcommand{\diam}{\operatorname{diam}}
\newcommand{\mean}[1]{\langle #1\rangle}

\def\Xint#1{\mathchoice
  {\XXint\displaystyle\textstyle{#1}}%
  {\XXint\textstyle\scriptstyle{#1}}%
  {\XXint\scriptstyle\scriptscriptstyle{#1}}%
  {\XXint\scriptscriptstyle\scriptscriptstyle{#1}}%
  \!\int}
\def\XXint#1#2#3{{\setbox0=\hbox{$#1{#2#3}{\int}$}
    \vcenter{\hbox{$#2#3$}}\kern-.5\wd0}}

\def\dashint{\Xint-}

\mathchardef\ordinarycolon\mathcode`\:
\mathcode`\:=\string"8000
\begingroup \catcode`\:=\active
  \gdef:{\mathrel{\mathop\ordinarycolon}}
\endgroup

\author{Hongjie Dong}
\author{Seongmin Jeon}

\address[H. Dong]{Division of Applied Mathematics
\newline\indent
Brown University
\newline\indent
182 George Street, Providence RI 02912, USA}
\email{hongjie\_dong@brown.edu}
\address[S. Jeon]{Department of Mathematics Education
\newline\indent
Hanyang University
\newline\indent
222 Wangsimni-ro, Seongdong-gu, Seoul 04763, Republic of Korea}
\email{seongminjeon@hanyang.ac.kr}

\title[Boundary estimates for elliptic operators in divergence form]{Boundary estimates for elliptic operators in divergence form with VMO coefficients}
\thanks{}

\allowdisplaybreaks

\begin{document}
\begin{abstract}
We establish boundary regularity estimates for elliptic systems in divergence form with VMO coefficients. Additionally, we obtain nondegeneracy estimates of the Hopf-Oleinik type lemma for elliptic equations. In both cases, the moduli of continuity are expressed in terms of the $L^p$-mean oscillations of the coefficients and data.
\end{abstract}

\maketitle

\section{Introduction}
The qualitative theory of partial differential equations has been extensively studied over time. In this paper, we focus on boundary regularity estimates and nondegeneracy estimates of the Hopf-Oleinik type lemma.

\subsection{Boundary regularity}
For elliptic equations in divergence form, it was shown in \cite{DonKim17} that solutions are $C^1$ if the coefficients satisfy the Dini mean oscillation (DMO) condition, which is a weaker requirement than the standard Dini condition. This regularity extends to the boundary if the boundary also satisfies DMO-type conditions; see \cite{DonJeoVit24}. The DMO condition appears to be almost optimal to ensure the $C^1$-regularity of solutions, since even harmonic functions in the $C^1$-domain may fail to be Lipschitz, see e.g., \cite{Saf08}. 

Our first main objective in this paper is to establish boundary and interior regularity estimates for elliptic systems in divergence form. We consider coefficients having vanishing mean oscillation (VMO), which is a strictly weaker requirement than continuity and the DMO conditions. In particular, we will derive an explicit modulus of continuity that involves only $L^p$-mean oscillations of them.

\subsection{Hopf-Oleinik type lemma}
Our second central result concerns the Hopf-Oleinik type lemma. The classical Hopf-Oleinik lemma states as follows: suppose $u$ is a positive harmonic function in $\Omega$. If $\Omega\in C^2$ and $u(x_0)=0$ for some $x_0\in \partial\Omega$, then $$
\liminf_{t\to0+}\frac{u(x_0+t\nu)}t>0,
$$
where $\nu$ is the inward unit normal to $\partial\Omega$ at $x_0$. Such result also holds for elliptic equations in non-divergence form with measurable coefficients when $\Omega$ is $C^{1,\text{Dini}}$. See \cite{Saf08}, the survey paper \cite{AN22}, and the references therein.

The lemma also applies to equations in divergence form under certain regularity conditions on the coefficients and the boundary. The necessary condition has been relaxed to DMO conditions; see \cite{RenSirSoa23} {\color{blue}and \cite{DonJeoVit24}}. However, as in the case of boundary regularity, the Hopf-Oleinik lemma fails when only continuity assumptions are imposed without DMO conditions; see \cite{Naz12} and \cite{AN22}.

Analogous to our boundary regularity result, we achieve the estimates of a Hopf-Oleinik type lemma under the VMO condition, and provide an explicit lower bound of the finite difference quotient expressed in terms of $L^p$-mean oscillations.

\subsection{Main results}
In this paper, we consider a domain $\Omega$ whose boundary can be expressed as a Lipschitz graph. More precisely, we assume
$$
\Omega=\{(x',x_n)\in\R^n\,:\, x_n>\g_\Omega(x')\},\quad \partial\Omega=\{(x',x_n)\in\R^n\,:\, x_n=\gamma_\Omega(x')\}
$$
for some Lipscthiz function $\g_\Omega:\R^{n-1}\to\R$ satisfying
\begin{align}
    \label{eq:domain-graph}
    \g_\Omega(0)=0\quad\text{and}\quad \D_{x'}\g_\Omega(0)=0.
\end{align}

Below we state our main results. For relevant notation, we refer to Section~\ref{sec:prel}.

\medskip

We first consider the following elliptic system
\begin{align}\label{eq:pde}
\begin{cases}
        D_\al(A^{\al\be}D_\be u)=\div\bff&\text{in }\Omega\cap B_1,\\
        u=g&\text{on }\partial\Omega\cap B_1,
\end{cases}
\end{align}
where $u=(u^1,\ldots,u^m)^{T}$, $m\ge1$, is a (column) vector-valued function. The coefficients $A^{\al\be}=(a^{\al\be}_{ij})^m_{i,j=1}$ are $m\times m$-matrices, $1\le \al,\be\le n$.

For $p>n$ and $p_0>1$, we let
\begin{align}\label{eq:mod}\begin{split}
\sigma(\rho)&:=\left(\|u\|_{L^2(\Omega\cap B_1)}+\|g\|_{L^2(\Omega\cap B_1)}+\int_\rho^1\frac{(\omega_{\bff,p}+\omega_{\D g,p}+\omega_{A,p}+\omega_{\D_{x'}\g_\Omega,p})(s)}sds\right)\\
&\qquad \cdot \exp\left(C\int_\rho^1\frac{(\omega_{A,p_0}+\omega_{\D_{x'}\g_\Omega,p_0})(s)}sds\right), \quad 0<\rho<1/2,
\end{split}\end{align}
where $C=C(n,m,\la,p,p_0,\|\D_{x'}\g_\Omega\|_{L^\infty(B_1')}, \omega_{A,1}+\omega_{\D_{x'}\g_\Omega,1})>0$.

Then our first main theorem is as follows.

\begin{theorem}
    \label{thm:reg}
Let $u\in H^1({\Omega\cap B_1;\R^m)}$ be a solution of \eqref{eq:pde}, and $p>n$ and $p_0>1$ with $1/2<\frac1{p_0}+\frac1p<1$. Assume that $A^{\al\be}$ and $\D_{x'}\g_\Omega\in L^\infty(B_1')$ are of VMO, and $A^{\al\be}$ satisfies \eqref{eq:assump-coeffi}. Suppose $\bff\in L^\infty(\Omega\cap B_1;\R^{m\times n})$ and $g\in W^{1,\infty}(\Omega\cap B_1;\R^m)$. Then, for any $x,y\in \Omega\cap B_{1/2}$ with $0<|x-y|<1/4$,
\begin{align}
    \label{eq:reg}
    \frac{|u(x)-u(y)|}{|x-y|}\le C\sigma(|x-y|)+\|\D g\|_{L^\infty(\Omega\cap B_1)},
\end{align}
where $C>0$ is a constant depending only on $n$, $m$, $p$, $p_0$, $\la$, $\|\bff\|_{L^\infty(\Omega\cap B_1)}$, $\|\D g\|_{L^{\infty}(\Omega\cap B_1)}$, $\|\D_{x'}\g_\Omega\|_{L^\infty(B_1')}$, and $\omega_{A,1}+\omega_{\D_{x'}\g_\Omega,1}$.
\end{theorem}

A similar result was recently obtained in \cite{Tor24} for non-divergence form elliptic equations when either $x$ or $y$ is on $\partial\Omega$, where a modulus of continuity of solutions is expressed in terms of the boundary of the domain and the data. However, it is noteworthy that their approach relies on boundary Harnack principles and comparison principles, which are not applicable in the case of systems considered in this paper.

\begin{remark}
Our result in Theorem~\ref{thm:reg} can be extended to the second-order elliptic system with lower-order terms:
\begin{align}\label{eq:pde-lot}
    \begin{cases}D_\al(A^{\al\be}D_\be u)+D_\al(B^\al u)+\hat B^\al D_\al u+Nu=\div\bff+h&\text{in }\Omega\cap B_1,\\
      u=g  &\text{on }\partial\Omega\cap B_1,
    \end{cases}
\end{align}
where $B^\al\in L^\infty(\Omega\cap B_1;\R^{m\times m})$, $\hat B^\al, N\in L^{n+\e}(\Omega\cap B_1;\R^{m\times m})$, and $h\in L^{n+\e}(\Omega\cap B_1;\R^{m})$ for some $\e>0$. Indeed, for any $q\in(1,\infty)$, we have $u\in W^{1,q}(\Omega\cap B_{3/4};\R^m)$ by the $W^{1,q}$ estimate. Then, the first equation in \eqref{eq:pde} can be rewritten as
$$
D_\al(A^{\al\be}D_\be u)=D_\al\tilde f_\al+\tilde h,
$$
where $\tilde f_\al:=f_\al-B^\al u\in L^\infty$ and $\tilde h:=h-\hat B^\al D_\al u-Nu\in L^{n+\e}$. The latter can be rewritten into $\div H$ for some $H\in W^{1,n+\e}(\Omega\cap B_{3/4})\subset L^\infty(\Omega\cap B_{3/4})$ by solving a divergence equation.
\end{remark}

\begin{remark}
    In Theorem~\ref{thm:reg}, we may assume $p_0\le p$. Then, by using \eqref{eq:mod-nondecr}, we can bound $\sigma$ in \eqref{eq:mod} by
    \begin{align*}
    \sigma(\rho)&\lesssim\left(\|u\|_{L^2(\Omega\cap B_1)}+\|g\|_{L^2(\Omega\cap B_1)}+\int_\rho^1\frac{(\omega_{\bff,p}+\omega_{\D g,p})(s)}sds\right)\\
&\qquad \cdot \exp\left(C\int_\rho^1\frac{(\omega_{A,p}+\omega_{\D_{x'}\g_\Omega,p})(s)}sds\right).
\end{align*}
\end{remark}

Our second primary result concerns the elliptic scalar equation
    \begin{align}\label{eq:Hopf-nonflat}
            \div(A\D u)=0\quad\text{in }\Omega\cap B_1.
    \end{align}
Note that $0\in \partial\Omega$ from \eqref{eq:domain-graph}. We let
$$
\tilde\Omega:=\{(x',x_n)\in\R^n\,:\, x_n>\g_\Omega(x')+1/4\}\cap B_{3/4}\Subset \Omega\cap B_1.
$$

\begin{theorem}\label{thm:Hopf}
Let $u\in H^1(\Omega\cap B_1)$ be a nonnegative solution of \eqref{eq:Hopf-nonflat}, $u(0)=0$, and $p_0>1$. Suppose $A$ and $\D_{x'}\g_{\Omega}\in L^{\infty}(B_1')$ are of VMO, and $A$ satisfies \eqref{eq:assump-coeffi-scalar}. Then, there are constants $C>0$ and $c>0$, depending only on $n$, $\lambda$, $p_0$, $\|\D_{x'}\g_\Omega\|_{L^\infty(B_1')}$, and 
$\omega_{A,1}+\omega_{\D_{x'}\g_\Omega, 1}$,
such that for all $t\in(0,1/2)$,
\begin{align}
    \label{eq:Hopf}
    \frac{u(t\nu)}{t}\ge c\|u\|_{L^2(\tilde\Omega)}\exp\left(-C\int_{t}^1\frac{(\omega_{A,p_0}+\omega_{\D_{x'}\g_\Omega,p_0})(s)}{s}ds\right),
\end{align}
where $\nu$ is the inward unit normal to $\Omega$ at $0$.
\end{theorem}

\begin{remark}
We note that by the proof of Theorem \ref{thm:Hopf}, the factor $\|u\|_{L^2(\tilde\Omega)}$ can be replaced with $\|u\|_{L^2(\Omega\cap B_{3/4})}$ if $u$ vanishes on $\partial\Omega\cap B_1$. Using the interior Harnack inequality, it can also be replaced with $u(\nu/2)$.
\end{remark}

Similar to Theorem~\ref{thm:reg}, a result analogous to Theorem~\ref{thm:Hopf} was established in \cite{Tor24} for non-divergence form equations. However, our argument is quite different from that in \cite{Tor24}, as we do not rely on the use of barrier functions.

\subsection{Structure of the paper}
The paper is organized as follows. In Section~\ref{sec:prel}, we introduce several definitions, assumptions, and preliminary lemmas  necessary for establishing our main results. Section~\ref{sec:main-thm} is devoted to the proofs of the main theorems. Specifically, in Section~\ref{sec:main-thm-1} we present the proof of Theorem~\ref{thm:reg} and in Section~\ref{sec:main-thm-2} we provide the proof of the nondegeneracy estimates of the Hopf-Oleinik type lemma (Theorem~\ref{thm:Hopf}).

\section{Preliminaries}\label{sec:prel}

Throughout the paper, we use the notation $x=(x',x_n)$ to denote a point in $\R^n$. For $x\in\R^n$ and $r>0$, we denote
\begin{align*}
    &B_r(x):=\{y\in\R^n\,:\,|y-x|<r\},\quad B_r^+(x):=B_r(x)\cap \{y_n>0\},\\
    & B'_r(x'):=\{y'\in\R^{n-1}\,:\, |y'-x'|<r\}.
\end{align*}
When $x=0$, we simply write $B_r(0)=B_r$, $B_r^+(0)=B_r^+$, and $B'_r(0)=B'_r$.

For a function $u$ and a domain $\cD$, by $\mean{u}_\cD$ we mean the integral mean value of $u$ over $\cD$. That is, 
$$
\mean{u}_\cD=\dashint_{\cD}u=\frac1{|\cD|}\int_\cD u.
$$

The relation $A\lesssim B$ can be understood as $A\le CB$ for some constant $C>0$.

When there is no confusion, we drop the function space notation in norms for vectorial functions, e.g., $\|u\|_{L^2(B_1^+)}=\|u\|_{L^2(B_1^+;\R^m)}$ for $u\in L^2(B_1^+;\R^m)$.

\medskip

Below are precise definitions for $L^p$-mean oscillation and vanishing mean oscillation.

\begin{definition}
    For $p\in[1,\infty)$ and a domain $\cD\subset\R^n$, let $g\in L^p(\cD)$. Then
    $$
    \omega_{g,p}(r):=\sup_{x_0\in\cD}\left(\dashint_{B_r(x_0)\cap \cD}\left|g(x)-\mean{g}_{B_r(x_0)\cap\cD}\right|^pdx\right)^{1/p},\quad 0<r<\frac{\text{diam}(\cD)}2,
    $$
    is the \emph{$L^p$-mean oscillation} of $g$ in $\cD$.

    We say that $g$ is \emph{of VMO} in $\cD$ if $\omega_{g,1}(r)\to0$ as $r\to0$.
\end{definition}

Due to Hölder's inequality, we have
\begin{align}
    \label{eq:mod-nondecr}
    p\longmapsto\omega_{g,p}\text{ is nondecreasing}.
\end{align}
Moreover, $r\longmapsto \omega_{g,p}(r)$ is almost nondecreasing, i.e., for some constants $C>c>0$, depending only on $n$ and $p$,
\begin{align}
    \label{eq:mod-alm-nondecr}
    c\,\omega_{g,p}(r)\le \omega_{g,p}(s)\le C\omega_{g,p}(r)
\end{align}
whenever $r/2\le s\le r<\frac{\diam(\cD)}2$. See e.g. \cite{Li17}.

\medskip

The following properties can be directly checked.

\begin{lemma}
    \label{lem:DMO-proper}
    Let $p\in[1,\infty)$ and $\cD$ be a domain in $\R^n$. If $f$ and $g$ are of $L^p$-DMO, then so are $f+g$ and $fg$. Moreover, 
    \begin{align*}
        &\omega_{f+g,p}\le \omega_{f,p}+\omega_{g,p},\\
        &\omega_{fg,p}\le C\left(\|f\|_{L^\infty(\cD)}\omega_{g,p}+\|g\|_{L^\infty(\cD)}\omega_{f,p}\right),
    \end{align*}
    where $C=C(p)>0$.
\end{lemma}

For $p\in[1,\infty)$, we say that $g$ is of $L^p$-DMO if $\omega_{g,p}$ is a Dini function, i.e.,
$$
\int_0^r\frac{\omega_{g,p}(\rho)}\rho d\rho<\infty.
$$
Clearly, the VMO condition is strictly weaker than the DMO condition.

\medskip

Next, we provide assumptions on the coefficients.

\begin{condition}
    [Conditions on coefficients for systems]
The coefficients $A^{\al\be}=(a^{\al\be}_{ij})^m_{i,j=1}$, $1\le\al,\be\le n$, satisfies for some constant $\la>0$
\begin{align}
\label{eq:assump-coeffi}
\begin{cases}
    \text{the Legendre-Hadamard condition: }\,\,\,\lambda|\xi|^2|\mu|^2\le a^{\al\be}_{ij}(x)\xi_\alpha\xi_\beta \mu_i\mu_j,\\
    \qquad\qquad\qquad\qquad\qquad\qquad\qquad\qquad\xi\in\R^{n},\,\,\, \mu\in\R^m,\,\,\,x\in \Omega\cap B_1, \\
    \text{the uniform boundedness: }\,\,\,|A^{\al\be}(x)|\le 1/{\lambda},\quad x\in \Omega\cap B_1.
\end{cases}
\end{align}
\end{condition}

\begin{condition}
    [Conditions on coefficients for scalar equations]
The coefficients $A=(a_{ij})^n_{i,j=1}$ satisfies for some constant $\la>0$
\begin{align}
\label{eq:assump-coeffi-scalar}
\begin{cases}
    \text{the uniform ellipticity: }\,\,\,\lambda|\xi|^2\le \mean{A(x)\xi,\xi},\quad\xi\in\R^n,\,\,\,x\in \Omega\cap B_1, \\
    \text{the uniform boundedness: }\,\,\,|A(x)|\le 1/\lambda,\quad x\in \Omega\cap B_1.
\end{cases}
\end{align}
\end{condition}

The following version of the Morrey inequality can be derived with a slight modification of the proof of Theorem~4 in Section 5.6.2 of \cite{Eva98}.

\begin{lemma}
    \label{lem:Morrey} 
    Suppose $p>n$ and $u\in W^{1,p}(B_2^+;\R^m)$. Then there is a constant $C>0$, depending only on $n$ and $p$, such that for any $x,y\in B_2^+$ with $r:=|x-y|>0$,
    $$
    \frac{|u(x)-u(y)|}{|x-y|}\le C\left(\dashint_{B_r^+(x)\cap B_2^+}|\D u|^p\right)^{1/p}.
    $$
\end{lemma}


\section{Proofs of main theorems}\label{sec:main-thm}

\subsection{Proof of Theorem~\ref{thm:reg}}\label{sec:main-thm-1}
The objective of this section is to establish Theorem~\ref{thm:reg}. We first consider a simpler case with a flat boundary and a zero Dirichlet condition:
\begin{align}\label{eq:pde-flat}
        \begin{cases}
            D_\al(A^{\al\be}D_\be u)=\div\bff&\text{in }B_2^+,\\
            u=0&\text{on }B_2'.
        \end{cases}
    \end{align}

Throughout this section, we fix $p>n$ and $p_0>1$ satisfying $1/2<\frac1{p_0}+\frac1p<1$, and let $p_1:=\frac{p_0p}{p_0+p}$ so that $1<p_1<2$. 

We say that a positive constant is universal if it depends only on $n,m,\la,p$, $p_0$, and $\omega_{A,1}$.

For $x_0\in \overline{B_1^+}$ and $r\in(0,1/2)$, we set
\begin{align*}
    \phi(x_0,r):=\begin{cases}\inf_{\bq\in\R^{m\times n}}\left(\dashint_{B_r(x_0)}|\D u-\bq|^{p_1}\right)^{1/p_1}&\text{if }r<(x_0)_n,\\
        \inf_{q\in\R^m}\left(\dashint_{B_r^+(x_0)}|(\D_{x'}u,D_nu-q)|^{p_1}\right)^{1/p_1}&\text{if }r\ge (x_0)_n.
    \end{cases}
\end{align*}
We take $\bq_{x_0,r}=(\bq'_{x_0,r}, (\bq_{x_0,r})_n)\in\R^{m\times (n-1)}\times \R^{m\times 1}$ such that
$$
\phi(x_0,r)=\left(\dashint_{B_r^+(x_0)}|\D u-\bq_{x_0,r}|^{p_1}\right)^{1/p_1}.
$$
Note that $\bq_{x_0,r}'=0$ when $r\ge (x_0)_n$. For a small constant $\ka\in(0,1/16)$ that will be determined after Lemma~\ref{lem:int-diff-est}, we let
$$
\Phi(x_0,r):=\sum_{j=0}^{j_0^r}\phi(x_0,\ka^{-j}r),\quad\text{where $j_0^r$ is the largest integer satisfying $\ka^{-j_0^r}r\le 1/2$}.
$$

\begin{lemma}\label{lem:grad-est}
    Let $u\in W^{1,p_1}(B_2^+;\R^m)$, $x_0\in \overline{B_{1/2}^+}$, and $r\in(0,1/8)$. Then
    \begin{align}\label{eq:grad-est}
        \left(\dashint_{B_r^+(x_0)}|\D u|^{p_1}\right)^{1/p_1}\le C\ka^{-\frac{n}{p_1}}\left(\Phi(x_0,r)+\|\D u\|_{L^{p_1}(B_1^+)}\right)
    \end{align}
    for some constant $C=C(n,p_1)>0$.
\end{lemma}

\begin{proof}
For simplicity, we write $j_0=j_0^r$. Then
\begin{align}
    \label{eq:grad-est-1}
    \begin{split}
        &\left(\dashint_{B_r^+(x_0)}|\D u|^{p_1}\right)^{1/p_1}\\
        &\le \left(\dashint_{B_r^+(x_0)}|\D u-\bq_{x_0,r}|^{p_1}\right)^{1/p_1}+\sum_{j=1}^{j_0}|\bq_{x_0,\ka^{-j+1}r}-\bq_{x_0,\ka^{-j}r}|+|\bq_{x_0,\ka^{-j_0}r}|.
    \end{split}
\end{align}
Here we assume that the second term vanishes when $j_0=0$.
Regarding the first term in the right-hand side in \eqref{eq:grad-est-1}, we simply have
\begin{align}
    \label{eq:grad-est-2}
    \left(\dashint_{B_r^+(x_0)}|\D u-\bq_{x_0,r}|^{p_1}\right)^{1/p_1}=\phi(x_0,r)\le\Phi(x_0,r).
\end{align}
Concerning the second term, we compute
\begin{align}
    \label{eq:grad-est-3}
    \begin{split}
    &\sum_{j=1}^{j_0}|\bq_{x_0,\ka^{-j+1}r}-\bq_{x_0,\ka^{-j}r}|\\
    &\le \sum_{j=1}^{j_0}\left(\left(\dashint_{B^+_{\ka^{-j+1}r}(x_0)}|\D u-\bq_{x_0,\ka^{-j+1}r}|^{p_1}\right)^{1/p_1}+\left(\dashint_{B^+_{\ka^{-j+1}r}(x_0)}|\D u-\bq_{x_0,\ka^{-j}r}|^{p_1}\right)^{1/p_1}\right)\\
    &\le \sum_{j=1}^{j_0}\left(\phi(x_0,\ka^{-j+1}r)+\ka^{-\frac{n}{p_1}}\phi(x_0,\ka^{-j}r)\right)\\
    &\le C\ka^{-\frac{n}{p_1}}\Phi(x_0,r).
\end{split}\end{align}
Finally, for the last term,
\begin{align}
    \label{eq:grad-est-4}
    \begin{split}
        &|\bq_{x_0,\ka^{-j_0}r}|\\
        &\le \left(\dashint_{B^+_{\ka^{-j_0}r}(x_0)}|\D u-\bq_{x_0,\ka^{-j_0}r}|^{p_1}\right)^{1/p_1}+\left(\dashint_{B^+_{\ka^{-j_0}r}(x_0)}|\D u|^{p_1}\right)^{1/p_1}\\
        &\le \phi(x_0,\ka^{-j_0}r)+C\ka^{-\frac{n}{p_1}}\|\D u\|_{L^{p_1}(B_1^+)}\le \Phi(x_0,r)+C\ka^{-\frac{n}{p_1}}\|\D u\|_{L^{p_1}(B_1^+)}.
    \end{split}
\end{align}
By combining \eqref{eq:grad-est-1}-\eqref{eq:grad-est-4}, we conclude \eqref{eq:grad-est}.    
\end{proof}

\begin{lemma}\label{lem:bdry-diff-est}
    Let $u:B_2^+\to\R^m$ be a $H^1$-weak solution of \eqref{eq:pde-flat}. Suppose that $A^{\al\be}$ is of VMO and satisfies \eqref{eq:assump-coeffi} in $B_2^+$, and $\bff\in L^\infty(B_2^+;\R^{m\times n})$. Then there exists a universal constant $C_1>0$ such that for any $\bar x_0\in B_{1/2}'$ and $r\in(0,1/8)$,
    \begin{align}
        \label{eq:bdry-diff-est}\begin{split}
        \phi(\bar x_0,\ka r)&\le C_1\ka\phi(\bar x_0,r)+C_1\ka^{-\frac{n}{p_1}}\omega_{\bff,p}(r)\\
        &\qquad+C_1\ka^{-\frac{2n}{p_1}}\left(\Phi(x_0,r)+\|\D u\|_{L^{p_1}(B_1^+)}\right)\omega_{A,p_0}(r).
    \end{split}\end{align}
\end{lemma}

\begin{proof}
Without loss of generality, we may assume that $\bar x_0=0$. We write for simplicity
$$
\bar A^{\al\be}:=\mean{A^{\al\be}}_{B_{r/2}^+},\quad \bar\bff:=\mean{\bff}_{B_{r/2}^+}
$$
and note that
\begin{align*}
    D_\al(\bar A^{\al\be}D_\be u)=D_\al(\bff_\al-\bar\bff_\al+(\bar A^{\al\be}-A^{\al\be})D_\be u).
\end{align*}
Let $\cD$ be a smooth and convex domain in $\R^n$ satisfying $B^+_{r/3}\subset\cD\subset B^+_{r/2}$. Let $w$ be the weak solution of
\begin{align*}
    \begin{cases}
        D_\al(\bar A^{\al\be}D_\be w)=D_\al(\bff_\al-\bar\bff_\al+(\bar A^{\al\be}-A^{\al\be})D_\be u)&\text{in }\cD,\\
        w=0&\text{on }\partial \cD.
    \end{cases}
\end{align*}
Then we have by the $W^{1,p_1}$ estimate and Hölder's inequality,
\begin{align*}
    \left(\dashint_{B_{r/4}^+}|\D w|^{p_1}\right)^{1/p_1}\lesssim \omega_{\bff,p_1}(r/2)+\omega_{A,p_0}(r/2)\left(\dashint_{B_{r/2}^+}|\D u|^{p}\right)^{1/p}.
\end{align*}
Moreover, since $u$ solves $D_\al(A^{\al\be}D_\be u)=\div\left(\bff-\mean{\bff}_{B_{r}^+}\right)$ in $B_{r}^+$ and $A^{\al\be}$ has vanishing mean oscillations, we have by the $W^{1,p}$ estimate and the boundary Poincaré inequality,
\begin{align*}
    \left(\dashint_{B_{r/2}^+}|\D u|^{p}\right)^{1/p}\lesssim \left(\dashint_{B_{r}^+}|\D u|^{p_1}\right)^{1/p_1}+\omega_{\bff,p}(r).
\end{align*}
Combining the previous two estimates and using \eqref{eq:mod-nondecr} and \eqref{eq:mod-alm-nondecr} yield
\begin{align}
    \label{eq:w-est-bdry}
    \begin{split}
        \left(\dashint_{B_{r/4}^+}|\D w|^{p_1}\right)^{1/p_1}&\lesssim \omega_{\bff,p_1}(r/2)+\omega_{A,p_0}(r/2)\left(\left(\dashint_{B_{r}^+}|\D u|^{p_1}\right)^{1/p_1}+\omega_{\bff,p}(r)\right)\\
        &\lesssim \omega_{\bff,p}(r)+\omega_{A,p_0}(r)\left(\dashint_{B_{r}^+}|\D u|^{p_1}\right)^{1/p_1}.
    \end{split}
\end{align}
Next, we observe that $v:=u-w$ satisfies
\begin{align*}
    \begin{cases}
        D_\al(\bar A^{\al\be}D_\be v)=0&\text{in }B_{r/4}^+,\\
        v=0&\text{on }B_{r/4}'.
    \end{cases}
\end{align*}
By the gradient estimate for elliptic systems with constant coefficients,
$$
\|Dv\|_{L^\infty(B^+_{r/8})}\lesssim \frac1r\left(\dashint_{B^+_{r/4}}|v|^{p_1}\right)^{1/p_1}.
$$
Since $\D_{x'}v$ satisfies the same system, we have
$$
\|DD_{x'}v\|_{L^\infty(B_{r/8}^+)}\lesssim \frac1r\left(\dashint_{B_{r/4}^+}|\D_{x'}v|^{p_1}\right)^{1/p_1}.
$$
Since 
$$
D_{nn}v=-(\bar A^{nn})^{-1}\sum_{(\al,\be)\neq(n,n)}\bar A^{\al\be}D_{\al\be}v,
$$
we further have
\begin{align*}
    \|D^2v\|_{L^\infty(B_{r/8}^+)}\lesssim \frac1r\left(\dashint_{B_{r/4}^+}|\D_{x'}v|^{p_1}\right)^{1/p_1}\lesssim\frac1r\left(\dashint_{B_{r/4}^+}|(\D_{x'}v, D_nv-q)|^{p_1}\right)^{1/p_1}
\end{align*}
for any $q\in\R^m$. This, along with the fact that $\D_{x'}v=0$ on $B'_{r/4}$, gives
\begin{align*}
    \left(\dashint_{B_{\ka r}^+}|(\D_{x'}v, D_nv-\mean{D_nv}_{B_{\ka r}^+})|^{p_1}\right)^{1/p_1}&\le 2\ka r\|D^2v\|_{L^\infty(B_{r/8}^+)}\\
    &\le C\ka\left(\dashint_{B_{r/4}^+}|(\D_{x'}v, D_nv-q)|^{p_1}\right)^{1/p_1}.
\end{align*}
By combining this with \eqref{eq:w-est-bdry} and using the triangle inequality, we obtain
\begin{align*}
    &\left(\dashint_{B_{\ka r}^+}|(\D_{x'}u,D_nu-\mean{D_nv}_{B_{\ka r}^+})|^{p_1}\right)^{1/p_1}\\
    &\le C\left(\dashint_{B_{\ka r}^+}|(\D_{x'}v,D_nv-\mean{D_nv}_{B_{\ka r}^+})|^{p_1}\right)^{1/p_1}+C\left(\dashint_{B_{\ka r}^+}|\D w|^{p_1}\right)^{1/p_1}\\
    &\le C\ka \left(\dashint_{B_{r/4}^+}|(\D_{x'}v,D_nv-q)|^{p_1}\right)^{1/p_1}+C\left(\dashint_{B_{\ka r}^+}|\D w|^{p_1}\right)^{1/p_1}\\
    &\le C\ka\left(\dashint_{B_{r/4}^+}|(\D_{x'}u,D_nu-q)|^{p_1}\right)^{1/p_1}+C\ka^{-\frac{n}{p_1}}\left(\dashint_{B_{r/4}^+}|\D w|^{p_1}\right)^{1/p_1}\\
    &\le C\ka\left(\dashint_{B_r^+}|(\D_{x'}u,D_nu-q)|^{p_1}\right)^{1/p_1}+C\ka^{-\frac{n}{p_1}}\omega_{\bff,p}(r)\\
    &\qquad+C\ka^{-\frac{n}{p_1}}\left(\dashint_{B_{r}^+}|\D u|^{p_1}\right)^{1/p_1}\omega_{A,p_0}(r).
\end{align*}
Since $q\in\R^m$ is arbitrary, \eqref{eq:bdry-diff-est} follows by applying Lemma~\ref{lem:grad-est}.    
\end{proof}

\begin{lemma}
    \label{lem:int-diff-est}
    Let $u$, $A^{\al\be}$, and $\bff$ be in Lemma~\ref{lem:bdry-diff-est}. Then there is a universal constant $C_2>0$ such that if $x_0\in B_{1/2}^+$ and $r\in(0,1/8)$ with $B_r(x_0)\subset B_1^+$, then
    \begin{align}
        \label{eq:int-diff-est}
        \begin{split}
            \phi(x_0,\ka r)&\le C_2\ka\phi(x_0,r)+C_2\ka^{-\frac{n}{p_1}}\omega_{\bff,p}(r)\\
            &\qquad+C_2\ka^{-\frac{2n}{p_1}}\left(\Phi(x_0,r)+\|\D u\|_{L^{p_1}(B_1^+)}\right)\omega_{A,p_0}(r).
        \end{split}
    \end{align}
\end{lemma}

\begin{proof}
We follow the argument in Lemma~\ref{lem:bdry-diff-est}. Let
$$
\bar A^{\al\be}:=\mean{A^{\al\be}}_{B_{r/2}(x_0)},\quad \bar\bff:=\mean{\bff}_{B_{r/2}(x_0)}.
$$
We then let $w$ be the solution of
\begin{align*}
    \begin{cases}
        D_\al(\bar A^{\al\be}D_\be w)=D_\al(\bff_\al-\bar\bff_\al+(\bar A^{\al\be}-A^{\al\be})D_\be u)&\text{in }B_{r/2}(x_0),\\
        w=0&\text{on }\partial B_{r/2}(x_0).
    \end{cases}
\end{align*}
By applying Hölder's inequality and the $W^{1,p_1}$ estimate, we get
\begin{align*}
    \left(\dashint_{B_{r/2}(x_0)}|\D w|^{p_1}\right)^{1/p_1}\lesssim \omega_{\bff,p_1}(r/2)+\omega_{A,p_0}(r/2)\left(\dashint_{B_{r/2}(x_0)}|\D u|^{p}\right)^{1/p}.
\end{align*}
Since $u-\mean{u}_{B_r(x_0)}$ satisfies $D_\al(A^{\al\be}D_\be(u-\mean{u}_{B_r(x_0)}))=\div(\bff-\mean{\bff}_{B_r(x_0)})$ in $B_r(x_0)$, we have by the $W^{1,p}$ estimate and Poincaré inequality that
$$
\left(\dashint_{B_{r/2}(x_0)}|\D u|^{p}\right)^{1/p}\lesssim \left(\dashint_{B_r(x_0)}|\D u|^{p_1}\right)^{1/p_1}+\omega_{\bff,p}(r).
$$
By combining the preceding two estimates and arguing as in the proof of Lemma~\ref{lem:bdry-diff-est}, we deduce
\begin{align}
    \label{eq:w-est-int}
    \left(\dashint_{B_{r/2}(x_0)}|\D w|^{p_1}\right)^{1/p_1}\lesssim \omega_{\bff,p}(r)+\omega_{A,p_0}(r)\left(\dashint_{B_r(x_0)}|\D u|^{p_1}\right)^{1/p_1}.
\end{align}
Moreover, $v:=u-w$ is a solution of the homogeneous equation
$$
D_\al(\bar A^{\al\be}D_\be v)=0\qquad\text{in }B_{r/2}(x_0).
$$
For any $\bq\in \R^{m\times n}$, $\D v-\bq$ satisfies the same equation, which gives
$$
\|D^2v\|_{L^\infty(B_{r/4}(x_0))}\le \frac{C}{r}\left(\dashint_{B_{r/2}(x_0)}|\D v-\bq|^{p_1}\right)^{1/p_1}.
$$
It follows that
$$
\left(\dashint_{B_{\ka r}(x_0)}|\D v-\mean{\D v}_{B_{\ka r}(x_0)}|^{p_1}\right)^{1/p_1}\le C\ka\left(\dashint_{B_{r/2}(x_0)}|\D v-\bq|^{p_1}\right)^{1/p_1}.
$$
As we saw in the proof of Lemma~\ref{lem:bdry-diff-est}, this estimate, along with \eqref{eq:w-est-int} and Lemma~\ref{lem:grad-est}, concludes \eqref{eq:int-diff-est}.    
\end{proof}

We fix a universal constant $\ka\in(0,1/16)$ small so that
$$
C_1\ka\le1/2\quad\text{and}\quad C_2\ka\le1/2,
$$
where $C_1$ and $C_2$ are as in Lemmas~\ref{lem:bdry-diff-est} and \ref{lem:int-diff-est}, respectively.

For $y_0\in \overline{B_{1/2}^+}$ and $\rho\in(0,1/8)$, we define 
$$
\Phi^*(y_0,\rho):=\Phi(y_0,\rho)-\frac12\Phi(y_0,\ka^{-1}\rho).
$$
Note that $\frac12\Phi(y_0,\rho)\le\Phi^*(y_0,\rho)\le\Phi(y_0,\rho)$.

Also, we let

$$
\eta(\rho):=\left(\|u\|_{L^2(B_2^+)}+\int_\rho^1\frac{\omega_{\bff,p}(s)}sds\right)\exp\left(C\int_\rho^1\frac{\omega_{A,p_0}(s)}sds\right),
$$
where $C>0$ is a universal constant to be specified.

The following proposition plays a key role in establishing our main result, Theorem~\ref{thm:reg}.

\begin{proposition}\label{prop:grad-est}
    Let $u$, $A^{\al\be}$, and $\bff$ be as in Lemma~\ref{lem:bdry-diff-est}. If $x_0\in B_{1/2}^+$ and $r\in(0,1/8)$, then
    \begin{align}\label{eq:gradient-est}
    \left(\dashint_{B_r^+(x_0)}|\D u|^{p_1}\right)^{1/p_1}\le C\eta(r)
    \end{align}
    for some universal constant $C>0$.
\end{proposition}

\begin{proof}
We write for simplicity
$$
d:=(x_0)_n>0.
$$
We observe that if $\rho\in(0,d)$, then Lemma~\ref{lem:int-diff-est}, along with our choice of $\ka$, gives
\begin{align*}
    \Phi(x_0,\ka \rho)-\Phi(x_0,\rho)&\le \frac12\left(\Phi(x_0,\rho)-\Phi(x_0,\ka^{-1}\rho)\right)+C\omega_{\bff,p}(\rho)\\
    &\qquad+C\left(\Phi(x_0,\rho)+\|\D u\|_{L^{p_1}(B_1^+)}\right)\omega_{A,p_0}(\rho),
\end{align*}
which yields
\begin{align}
    \label{eq:Phi-tilde-est-int}
    \Phi^*(x_0,\ka\rho)&\le (1+C\omega_{A,p_0}(\rho))\Phi^*(x_0,\rho)+C\omega_{\bff,p}(\rho)+C\|\D u\|_{L^{p_1}(B_1^+)}\omega_{A,p_0}(\rho).
\end{align}
Similarly, if $\bar x_0:=(x_0',0)\in B'_{1/2}$ and $0<\rho<1/8$, then we use Lemma~\ref{lem:bdry-diff-est} and argue as above to obtain
\begin{align}
    \label{eq:Phi-tilde-est-bdry}
    \Phi^*(\bar x_0,\ka\rho)&\le (1+C\omega_{A,p_0}(\rho))\Phi^*(\bar x_0,\rho)+C\omega_{\bff,p}(\rho)+C\|\D u\|_{L^{p_1}(B_1^+)}\omega_{A,p_0}(\rho).
\end{align}
We then split the proof into the following three cases:\\
Case A. $d<\frac{\ka^3}2$ and $r<\frac{\ka^3}2$,\\
Case B. $r\ge\frac{\ka^3}2$,\\
Case C. $d\ge\frac{\ka^3}2$ and $r<\frac{\ka^3}2$.\\

Case A. We first consider the case when $d<\frac{\ka^3}2$ and $r<\frac{\ka^3}2$. Recall that $j_0:=j_0^r$ is the largest nonnegative integer satisfying $\ka^{-j_0}r\le 1/2$, and note that $j_0\ge3$. We further divide into two subcases either $r\le d$ or $r>d$.

\smallskip

Case A-1. Suppose $r\le d$. Let $j_1=j_1(r,x_0)$ be the largest nonnegative integer satisfying $\ka^{-j_1}r\le d$. Then, from
$$
\ka^{-j_1}r\le d<\frac{\ka^3}2<\ka^{-j_0+2}r,
$$
we see $j_1<j_0-2$. For $j=j_1,j_1+1,\ldots, j_0-2$, we have $\ka^{-j}r+d\le \ka^{-j}r+\ka^{-j-1}r\le \ka^{-j-2}r$, thus $B^+_{\ka^{-j}r}(x_0)\subset B^+_{\ka^{-j-2}r}(\bar x_0)$, where $\bar x_0=(x'_0,0)$ as before. Thus, for such $j$,
\begin{align*}
    \phi(x_0,\ka^{-j}r)&\le \left(\dashint_{B^+_{\ka^{-j}r}(x_0)}|\D u-\bq_{\bar x_0, \ka^{-j-2}r}|^{p_1}\right)^{1/p_1}\\
    &\le C\left(\dashint_{B_{\ka^{-j-2}r}^+(\bar x_0)}|\D u-\bq_{\bar x_0,\ka^{-j-2}r}|^{p_1}\right)^{1/p_1}= C\phi(\bar x_0,\ka^{-j-2}r).
\end{align*}
Moreover, by using $\ka^{-j_0+1}r>\frac{\ka^2}2$, we have
\begin{align}\label{eq:phi-est}
\phi(x_0,\ka^{-j_0+1}r)+\phi(x_0,\ka^{-j_0}r)\le C\|\D u\|_{L^{p_1}(B_1^+)}.
\end{align}
Thus,
\begin{align*}
    \Phi^*(x_0,\ka^{-j_1}r)&\le \Phi(x_0,\ka^{-j_1}r)=\sum_{j=j_1}^{j_0-2}\phi(x_0,\ka^{-j}r)+\phi(x_0,\ka^{-j_0+1}r)+\phi(x_0,\ka^{-j_0}r)\\
    &\le C\sum_{j=j_1}^{j_0-2}\phi(\bar x_0,\ka^{-j-2}r)+C\|\D u\|_{L^{p_1}(B_1^+)}\\
    &\le C\Phi(\bar x_0,\ka^{-j_1-2}r)+C\|\D u\|_{L^{p_1}(B_1^+)}\\
    &\le C\Phi^*(\bar x_0,\ka^{-j_1-2}r)+C\|\D u\|_{L^{p_1}(B_1^+)}.
\end{align*}
By using this estimate and applying \eqref{eq:Phi-tilde-est-int} with $\rho=\ka^{-1}r,\ka^{-2}r,\ldots, \ka^{-j_1}r$ and \eqref{eq:Phi-tilde-est-bdry} with $\rho=\ka^{-j_1-3}r, \ka^{j_1-4}r,\ldots,\ka^{-j_0}$, we deduce
\begin{align}\label{eq:Phi-tilde-est}\begin{split}
    &\Phi^*(x_0,r)\\
    &\le C\left[\Phi^*(\bar x_0,\ka^{-j_0}r)+\|\D u\|_{L^{p_1}(B_1^+)}+\sum_{j=1}^{j_0}\left(\omega_{\bff,p}(\ka^{-j}r)+\|\D u\|_{L^{p_1}(B_1^+)}\omega_{A,p_0}(\ka^{-j}r)\right)\right]\times\\
    &\qquad \times\prod_{j=1}^{j_0}(1+C\omega_{A,p_0}(\ka^{-j}r))\\
    &\le C \left(\|\D u\|_{L^{p_1}(B_1^+)}+\int_r^1\frac{\omega_{f,p}(s)}sds
+\|\D u\|_{L^{p_1}(B_1^+)}\int_r^1\frac{\omega_{A,p_0}(s)}sds\right)\times\\
&\qquad\times\exp\left(C\int_r^1\frac{\omega_{A,p_0}(s)}sds\right)\\
    &\le C\left(\|\D u\|_{L^{p_1}(B_1^+)}+\int_r^1\frac{\omega_{f,p}(s)}sds\right)\exp\left(C\int_r^1\frac{\omega_{A,p_0}(s)}sds\right).
\end{split}\end{align}
In addition, since $p_1<2$ and $u$ solves $D_\al(A^{\al\be}D_\be u)=\div\left(\bff-\mean{\bff}_{B_2^+}\right)$ in $B_2^+$, we have by applying Hölder's inequality, Caccioppoli's inequality, \eqref{eq:mod-nondecr}, and \eqref{eq:mod-alm-nondecr} that
\begin{align}\label{eq:Hol-Cac}
    \|\D u\|_{L^{p_1}(B_1^+)}\le C\left(\|u\|_{L^{2}(B_2^+)}+\omega_{\bff,2}(2)\right)\le C\left(\|u\|_{L^{2}(B_2^+)}+\omega_{\bff,p}(1)\right).
\end{align}
This, together with \eqref{eq:Phi-tilde-est} and the inequality $\Phi^*(x_0,r)\ge\frac12\Phi(x_0,r)$, yields $\Phi(x_0,r)\lesssim\eta(r)$. By combining this with Lemma~\ref{lem:grad-est} and \eqref{eq:Hol-Cac}, we conclude \eqref{eq:gradient-est}.

\medskip

Case A-2. Suppose $r>d$. Then we have
$$
\phi(x_0,\ka^{-j}r)\le C\phi(\bar x_0,\ka^{-j-2}r)
$$
for $j=0,1,\ldots,j_0-2$. This, along with \eqref{eq:phi-est}, gives
$$
\Phi^*(x_0,r)\le C\left(\Phi^*(\bar x_0,\ka^{-2}r)+\|\D u\|_{L^{p_1}(B_1^+)}\right).
$$
By using this estimate and applying \eqref{eq:Phi-tilde-est-bdry} with $\rho=\ka^{-3}r,\ka^{-4}r,\ldots, \ka^{-j_0}r$, we get \eqref{eq:Phi-tilde-est}. We then argue as in Case A-1 to obtain \eqref{eq:gradient-est}.\\

Case B. If $r\ge \frac{\ka^3}2$, then we apply \eqref{eq:Hol-Cac} to obtain
$$
\left(\dashint_{B_r^+(x_0)}|\D u|^{p_1}\right)^{1/p_1}\le C\|\D u\|_{L^{p_1}(B_1^+)}\le C\eta(r).
$$
\smallskip

Case C. Suppose $d\ge \frac{\ka^3}2$ and $r<\frac{\ka^3}2$. Note that $r<d$ and let $j_1\ge 0$ be as in Case A-1. Then we have $\ka^{-j_1-1}r>d\ge\frac{\ka^3}2$, thus $\ka^{-j_1}r>\frac{\ka^4}2$. If $j_1=0$, then we argue as in Case B to get \eqref{eq:gradient-est}. Otherwise, we apply \eqref{eq:Phi-tilde-est-int} with $\rho=\ka^{-1}r,\ka^{-2}r,\ldots,\ka^{-j_1}r$ and argue as in Case A to get \eqref{eq:gradient-est}. This completes the proof.
\end{proof}

Now we are ready to prove Theorem~\ref{thm:reg}.

\begin{proof}[Proof of Theorem~\ref{thm:reg}]
Since $u$ solves \eqref{eq:pde}, it is easily seen that $\tilde u:=u-g$ satisfies
\begin{align*}
    \begin{cases}
        D_\al(A^{\al\be}D_\be\tilde u)=\div\tilde\bff&\text{in }\Omega\cap B_1,\\
        \tilde u=0&\text{on }\partial\Omega\cap B_1,
    \end{cases}
\end{align*}
where $\tilde\bff:=\mean{\tilde f_1,\ldots,\tilde f_n}$ with $\tilde f_\al:=f_\al-A^{\al\be}D_\be g$, $1\le\al\le n$.

Next, we consider the standard local diffeomorphism that flattens the boundary $\partial\Omega$
\begin{align}\label{eq:diff-flat}
\Psi(x):=(x',x_n+\g_\Omega(x')).
\end{align}
We let $J_{\Psi^{-1}}$ be the Jacobian matrix associated with $\Psi^{-1}$. Then a direct calculation gives that $\hat u:=\tilde u\circ\Psi$ satisfies (possibly after a dilation)
$$
\begin{cases}
    D_\al(\hat A^{\al\be}D_\be\hat u)=\div\hat\bff&\text{in }B_2^+,\\
    \hat u=0&\text{on }B_2',
\end{cases}
$$
where $\hat\bff:=(\hat f^1,\ldots,\hat f^m)^T$ with $\hat f^i:=(\tilde f^i J_{\Psi^{-1}})\circ \Psi$ and $n\times n$-matrices $\hat A_{ij}:=(J_{\Psi^{-1}}^T A_{ij} J_{\Psi^{-1}})\circ \Psi$ for each $1\le i,j\le m$. 

Now, we let $x,y\in B_{1/2}^+$ with $r:=|x-y|\in(0,1/4)$. Then we apply Lemma~\ref{lem:Morrey}, the $W^{1,p}$ estimate, Poincaré inequality, and Proposition~\ref{prop:grad-est} to get
\begin{align*}
    \frac{|\hat u(x)-\hat u(y)|}{|x-y|}&\lesssim \left(\dashint_{B_r^+(x)}|\D\hat u|^p\right)^{1/p}\\
    &\lesssim \left(\dashint_{B_{2r}^+(x)}|\D\hat u|^{p_1}\right)^{1/p_1}+\omega_{\hat\bff,p}(2r)\\
    &\lesssim \left(\|\hat u\|_{L^2(B_2^+)}+\int_r^1\frac{\omega_{\hat\bff,p}(s)}sds\right)\exp\left(C\int_r^1\frac{\omega_{\hat A,p_0}(s)}sds\right).
\end{align*}
By using $\hat u=(u-g)\circ\Psi$ and definitions of $\hat\bff$ and $\hat A$ and applying Lemma~\ref{lem:DMO-proper}, we conclude \eqref{eq:reg}. 
\end{proof}


\subsection{Proof of Theorem~\ref{thm:Hopf}}\label{sec:main-thm-2} In this section, we give the proof of our second main result, Theorem~\ref{thm:Hopf}.

The constants $C$ and $c$ throughout this section may vary from line to line, but depend only on $n$, $\la$, $p_0$, and $\omega_{A,1}$.

\begin{proof}[Proof of Theorem~\ref{thm:Hopf}]
We divide the proof into two steps.

\medskip\noindent\emph{Step 1.} We first consider the simple case with a flat boundary and a zero boundary condition:
\begin{align*}
    \begin{cases}
        \div(A\D u)=0&\text{in }B_1^+,\\
        u=0&\text{on }B_1'.
    \end{cases}
\end{align*}
The aim of this step is to show that
\begin{align}
    \label{eq:lower-bound}
    \frac{u(x_n)}{x_n}\ge c\|u\|_{L^2(B_{3/4}^+)}\exp\left(-C\int_{x_n}^1\frac{\omega(s)}sds\right),\quad 0<x_n<1/2.
\end{align}
Let $\cD$ be a smooth and convex domain in $\R^n$ satisfying $B^+_{2/3}\subset \cD\subset B^+_{3/4}$. For any $k\in\mathbb{N}$, we write $r_k:=2^{-k}$ and $\cD_k:=r_k\cD$. We also write for simplicity $\omega=\omega_{A,p_0}$. Let $w_k$ be a solution of 
\begin{align*}
    \begin{cases}
        \div(\mean{A}_{\cD_k}\D w_k)=\div((\mean{A}_{\cD_k}-A)\D u)&\text{in }\cD_k,\\
        w_k=0&\text{on }\partial \cD_k.
    \end{cases}
\end{align*}
Then $v_k:=u-w_k$ satisfies
\begin{align*}
    \begin{cases}
        \div(\mean{A}_{\cD_k}\D v_k)=0&\text{in }\cD_k,\\
        v_k=u&\text{on }\partial\cD_k.
    \end{cases}
\end{align*}
Note that since $u\ge0$, $v_k\ge0$ by the maximum principle. We take $p_1,p_2\in\R$ such that $1<p_1<p_0<p_2$ and $\frac{1}{p_0}+\frac1{p_2}=\frac1{p_1}$. Then we have by applying Hölder's inequality, the $W^{1,p_2}$ estimate, Carleson type estimate, and the Sobolev inequality that
\begin{align}\label{eq:w-grad-est}\begin{split}
    \left(\dashint_{\cD_k}|\D w_k|^{p_1}\right)^{1/p_1}&\le C\left(\dashint_{\cD_k}|(\mean{A}_{\cD_k}-A)\D u|^{p_1}\right)^{1/p_1}\\
    &\le C\omega(r_k)\left(\dashint_{\cD_k}|\D u|^{p_2}\right)^{1/p_2}\le C\frac{\omega(r_k)}{r_k}\left(\dashint_{2\cD_k}|u|^{p_2}\right)^{1/p_2}\\
    &\le C\frac{\omega(r_k)}{r_k}\inf_{x\in \cD_{k+1},x_n>r_{k+2}}u \\
    &\le C\frac{\omega(r_k)}{r_k}\left(\inf_{x\in \cD_{k+1},x_n>r_{k+2}}|w_k|+\sup_{\cD_{k+1}}v_k\right)\\
    &\le C_0\omega(r_k)\left(\dashint_{\cD_k}|\D w_k|^{p_1}\right)^{1/p_1}+C\frac{\omega(r_k)}{r_k}\sup_{\cD_{k+1}}v_k.
\end{split}\end{align}
Here, for the Carleson type estimate, we refer to \cite{CafFabMorSal81}*{Theorem~1.1}. Although the theorem is stated under the symmetry assumption on the coefficient matrix $A$, it is easily seen that the proof works without the symmetry condition.

For $C_0$ as in \eqref{eq:w-grad-est}, since $A$ is bounded and of VMO, we can take $k_0\in\mathbb{N}$ such that $\omega(r_k)<\frac1{2C_0}$ for every $k\ge k_0$. Then, we have by \eqref{eq:w-grad-est},
\begin{align}
    \label{eq:w-Lp-est}
    \frac{\inf_{x\in \cD_{k+1},x_n>r_{k+2}}|w_k|}{r_k}\lesssim \left(\dashint_{\cD_k}|\D w_k|^{p_1}\right)^{1/p_1}\lesssim \frac{\omega(r_k)}{r_k}\sup_{\cD_{k+1}}v_k, \quad k\ge k_0.
\end{align}
We denote
$$
\mu_k:=\inf_{\cD_{k+1}}\frac{v_k}{x_n}.
$$
Then, by applying the second inequality in \eqref{eq:w-Lp-est} and the boundary Harnack principle, we get
\begin{align}
    \label{eq:w-Lp-est-2}
    \left(\dashint_{\cD_k}|\D w_k|^{p_1}\right)^{1/p_1}\lesssim \omega(r_k)\sup_{\cD_{k+1}}\frac{v_k}{x_n}\lesssim \omega(r_k)\mu_k.
\end{align}
Moreover, we have by the boundary estimates for elliptic equations with constant coefficients and the boundary Harnack principle that for $k\in\mathbb{N}$,
\begin{align}
    \label{eq:v-sup-est}
    \begin{aligned}
           &2^{-k}\|D^2 v_k\|_{L^\infty(\cD_{k+1})}+  \|\D v_k\|_{L^\infty(\cD_{k+1})}\lesssim \frac{1}{r_k}\left(\dashint_{\frac12(\cD_k+\cD_{k+1})}|v_k|^{p_1}\right)^{1/p_1}\\
           &\lesssim \sup_{\frac12(\cD_k+\cD_{k+1})}\frac{v_k}{x_n}\lesssim \mu_{k}.
    \end{aligned}
\end{align}
Now, in $\cD_{k+1}$, we further decompose $v_{k}=\tilde w_{k}+\tilde v_{k}$, where $\tilde w_{k}$ solves
\begin{align*}
    \begin{cases}
        \div(\mean{A}_{\cD_{k+1}}\D\tilde w_{k})=\div((\mean{A}_{\cD_{k+1}}-\mean{A}_{\cD_{k}})\D v_{k})&\text{in }\cD_{k+1},\\
        \tilde w_{k}=0&\text{on }\partial\cD_{k+1}.
    \end{cases}
\end{align*}
By using the Schauder estimate and \eqref{eq:v-sup-est}, we obtain
\begin{align}
    \label{eq:tilde-w-sup-est}
    \|\D\tilde w_{k}\|_{L^\infty(\cD_{k+1})}+\frac1{r_k}\|\tilde w_{k}\|_{L^\infty(\cD_{k+1})}\lesssim \omega(r_k)\mu_{k}.
\end{align}
We denote
$$
\tilde\mu_k:=\inf_{\cD_{k+2}}\frac{\tilde v_k}{x_n},\quad k\in\mathbb{N}.
$$
Since both $v_{k+1}$ and $\tilde v_{k}$ satisfy the equation
\begin{align*}
    \begin{cases}
        \div(\mean{A}_{\cD_{k+1}}\D v)=0&\text{in }\cD_{k+1},\\
        v=0&\text{on }\overline{\cD_{k+1}}\cap \{x_n=0\},
    \end{cases}
\end{align*}
we have by using the boundary elliptic estimate, the triangle inequality, and Poincaré inequality that
\begin{align*}
    \tilde\mu_{k}-\mu_{k+1}&=\inf_{\cD_{k+2}}\frac{\tilde v_k}{x_n}-\inf_{\cD_{k+2}}\frac{v_{k+1}}{x_n}\le \sup_{\cD_{k+2}}\frac{|\tilde v_{k}-v_{k+1}|}{x_n}\\
    &\lesssim \sup_{\cD_{k+2}}|\D(\tilde v_{k}-v_{k+1})|\lesssim \frac1{r_k}\left(\dashint_{\cD_{k+1}}|\tilde v_k-v_{k+1}|^{p_1}\right)^{1/p_1}\\
    &\lesssim \frac1{r_k}\left(\dashint_{\cD_{k+1}}|v_k-v_{k+1}|^{p_1}\right)^{1/p_1}+\frac1{r_k}\left(\dashint_{\cD_{k+1}}|\tilde w_{k}|^{p_1}\right)^{1/p_1}\\
    &=\frac1{r_k}\left(\dashint_{\cD_{k+1}}|w_k-w_{k+1}|^{p_1}\right)^{1/p_1}+\frac1{r_k}\left(\dashint_{\cD_{k+1}}|\tilde w_{k}|^{p_1}\right)^{1/p_1}\\
    &\lesssim \left(\dashint_{\cD_{k+1}}|\D(w_k-w_{k+1})|^{p_1}\right)^{1/p_1}+\frac1{r_k}\left(\dashint_{\cD_{k+1}}|\tilde w_{k}|^{p_1}\right)^{1/p_1}\\
    &\lesssim \omega(r_k)\mu_k+\omega(r_{k+1})\mu_{k+1},
\end{align*}
where we applied \eqref{eq:w-Lp-est-2} and \eqref{eq:tilde-w-sup-est} in the last step. On the other hand, by using \eqref{eq:tilde-w-sup-est} and that $w_k=0$ on $\partial\cD_k$, we get
\begin{align*}
    \tilde\mu_k\ge \inf_{\cD_{k+2}}\frac{v_k}{x_n}-\sup_{\cD_{k+2}}\frac{|\tilde w_k|}{x_n}\ge \mu_k-\sup_{\cD_{k+2}}|\D\tilde w_k|\ge \mu_k-C\omega(r_k)\mu_k.
\end{align*}
By combining the preceding two estimates, we deduce
$$
(1+C\omega(r_{k+1}))\mu_{k+1}\ge (1-C\omega(r_k))\mu_k,
$$
and thus,
$$
\mu_{k+1}\ge (1-C\omega(r_k))\mu_k,\quad k\ge k_0.
$$
By iteration, we further have for $k\ge k_0$,
\begin{align}
    \label{eq:v-lower-bound}
    \mu_k\ge \mu_{k_0}\exp\left(-C\int_{r_k}^1\frac{\omega(s)}sds\right).
\end{align}
To find a lower bound for $\mu_{k_0}$, we use \eqref{eq:v-sup-est}, \eqref{eq:tilde-w-sup-est} and the Carleson type estimate to get
\begin{align*}
    C\mu_{k_0}&\ge \frac1{r_{k_0}}\left(\dashint_{\cD_{k_0}}|v_{k_0}|^{p_1}\right)^{1/p_1}\ge\frac1{r_{k_0}}\left(\dashint_{\cD_{k_0}}|u|^{p_1}\right)^{1/p_1}-\frac1{r_{k_0}}\left(\dashint_{\cD_{k_0}}|w_{k_0}|^{p_1}\right)^{1/p_1}\\
    &\ge c\|u\|_{L^2(B_{3/4}^+)}-C\omega(r_{k_0})\mu_{k_0}.
\end{align*}
This gives 
$$
\mu_{k_0}\ge c\|u\|_{L^2(B_{3/4}^+)},
$$
which combined with \eqref{eq:v-lower-bound} yields
\begin{align}
    \label{eq:v-lower-bound-1}
    \mu_k\ge c\|u\|_{L^2(B_{3/4}^+)}\exp\left(-C\int_{r_k}^1\frac{\omega(s)}sds\right).
\end{align}
Next, we use the first inequality in \eqref{eq:w-Lp-est} and \eqref{eq:w-Lp-est-2} to have
\begin{align*}
    \frac{\inf_{x\in\cD_{k+1}, x_n>r_{k+2}}|w_k|}{r_k}\lesssim \omega(r_k)\mu_k.
\end{align*}
Take $x_k\in\{x\in\overline{\cD_{k+1}}\,:\, x_n\ge r_{k+2}\}$ such that
$
|w_k(x_k)|=\inf_{x\in\cD_{k+1}, x_n>r_{k+2}}|w_k|.
$
Then, by the above estimate,
\begin{align*}
    \frac{u(x_k)}{r_{k+1}}\ge \frac{v_k(x_k)}{r_{k+1}}-\frac{|w_k(x_k)|}{r_{k+1}}\ge \frac{v_k(x_k)}{r_{k+1}}-C\omega(r_k)\mu_k\ge (1-C\omega(r_k))\frac{v_k(x_k)}{r_{k+1}}.
\end{align*}
For $k\ge k_0$ large so that $1-C\omega(r_k)\ge1/2$, we have by the previous estimate and \eqref{eq:v-lower-bound-1},
$$
\frac{u(x_k)}{r_{k+1}}\ge c\|u\|_{L^2(B_{3/4}^+)}\exp\left(-C\int_{r_k}^1\frac{\omega(s)}sds\right).
$$
Finally, by applying the Harnack inequality, we conclude \eqref{eq:lower-bound}. 

\medskip\noindent\emph{Step 2.}
In this step, we consider the general case, where $u$ is a solution of \eqref{eq:Hopf-nonflat}. 
We  use \eqref{eq:diff-flat} to straighten the boundary $\partial\Omega$ as in the proof of Theorem~\ref{thm:reg}:
\begin{align*}
    \div(\hat A\D\hat u)=0\quad\text{in }B_1^+,
\end{align*}
where $\hat u:=u\circ\Psi$ and $\hat A:=(J^T_{\Psi^{-1}}AJ_{\Psi^{-1}})\circ\Psi$. 

We consider a smooth cut-off function $\zeta=\zeta(x_n)$ satisfying $0\le\zeta\le1$, $\zeta=1$ when $x_n\ge 1/4$, and $\zeta=0$ when $x_n\le 1/8$. We let $\bar u$ be the solution of
\begin{align*}
    \begin{cases}
        \div(\hat A\D\bar u)=0&\text{in }B_{7/8}^+,\\
        \bar u=\hat u\,\zeta&\text{on }\partial B_{7/8}^+.
    \end{cases}
\end{align*}
Then $\bar u=0$ on $B_{7/8}'$ and, by the comparison principle, 
\begin{equation}
    \label{eq0922}
0\le\bar u\le \hat u\quad \text{in}\, B^+_{7/8}.
\end{equation}
By applying \eqref{eq:lower-bound} to $\bar u$, we get
\begin{align}
    \label{eq:lower-bound-2}
    \frac{\bar u(x_n)}{x_n}\ge c\|\bar u\|_{L^2(B_{3/4}^+)}\exp\left(-C\int_{x_n}^1\frac{\omega(s)}sds\right),\quad 0<x_n<1/2.
\end{align}

To rewrite \eqref{eq:lower-bound-2} in terms of $u$, we define
$$
S:=\partial B_{7/8-\rho_0}\cap\{x_n>1/4\}
$$
for some $\rho_0\in (0,1/20)$ to be chosen later. We claim that for some $C=C(n,\la)>0$,
\begin{align}
    \label{eq:sup-inf}
    \sup_S\hat u\le C\inf_S\bar u.
\end{align}
To prove it, we write $M:=\sup_S\hat u$, and note that $cM\le\hat u\le CM$ in $B_{15/16}^+\cap\{x_n\ge 1/8\}$ by the Harnack inequality. Given $x\in S$, we let $x^*:=\frac{7/8}{7/8-\rho_0}x$ so that $x^*\in\partial B_{7/8}\cap \{x_n>1/4\}$ and $|x^*-x|=\rho_0$. By the boundary Hölder estimate for $\bar u$, the interior Hölder estimate for $\hat u$, and \eqref{eq0922}, we have for some $\de=\de(n,\la)\in(0,1)$ that
\begin{align*}
    [\bar u]_{C^\de(B_{1/20}(x^*)\cap B_{7/8}^+)}&\le C\left(\|\bar u\|_{L^2(B_{1/10}(x^*)\cap B_{7/8}^+)}+[\hat u]_{C^\de(B_{1/10}(x^*)\cap B_{7/8}^+)}\right)\\
    &\le C\left(\|\hat u\|_{L^2(B_{1/10}(x^*)\cap B_{7/8}^+)}+\sup_{B_{1/8}(x^*)\cap B_{15/16}^+}\hat u\right)\\
    &\le CM.
\end{align*}
It follows that
\begin{align*}
    \bar u(x)&\ge \bar u(x^*)-[\bar u]_{C^\de(B_{1/20}(x^*)\cap B_{7/8}^+)}|x-x^*|^\de\ge\hat u(x^*)-CM\rho_0^\de\ge c_0M-C_0M\rho_0^\de.
\end{align*}
Here, $c_0>0$ and $C_0>0$ depend only on $n$ and $\la$. By taking $\rho_0:=\left(\frac{c_0}{2C_0}\right)^{1/\de}$, we get $\bar u(x)\ge \frac{c_0}2M$. Since $x\in S$ is arbitrary, \eqref{eq:sup-inf} is proved.

We now use \eqref{eq:sup-inf} and apply the Harnack inequality to get
\begin{align*}
    \|\hat u\|_{L^2(B^+_{3/4}\cap \{x_n>1/4\})}\lesssim \sup_S\hat u\lesssim \inf_S\bar u\lesssim \|\bar u\|_{L^2(B^+_{3/4})}.
\end{align*}
By using this inequality, \eqref{eq:lower-bound-2}, and the definitions of $\bar u$ and $\hat A$, we conclude Theorem~\ref{thm:Hopf}.
\end{proof}

\section*{Acknowledgment}
H. Dong was partially supported by the NSF under agreement DMS-2350129. S. Jeon was supported by the research fund of Hanyang University(HY-202400000003278). The authors would like to thank Stefano Vita for helpful discussions.


\end{document}